\documentclass[10pt,reqno]{amsart}
\usepackage{amsmath,multicol}
\usepackage{amsthm}
\usepackage{amssymb}
\usepackage{amsfonts}
\usepackage{subcaption}
\usepackage{graphicx}
\usepackage{latexsym}
\usepackage{cite, hyperref, xcolor}
\usepackage{stmaryrd}
\usepackage{enumitem}

\usepackage[font=footnotesize,labelfont=bf]{caption}

\usepackage[sc]{mathpazo}
\usepackage[T1]{fontenc}

\usepackage{tikz,amsmath,amsfonts}
\usetikzlibrary{positioning}
\usetikzlibrary{decorations.pathreplacing}

\newcommand{\NN}{\mathbb{N}}
\newcommand{\ZZ}{\mathbb{Z}}
\newcommand{\RR}{\mathbb{R}}
\newcommand{\MM}{\mathsf{M}} 
\newcommand{\Le}{\mathsf{L}}

\newcommand{\mm}{\mathsf{m}}
\newcommand{\QQ}{\mathbb{Q}}

\providecommand{\multi}[1]{\llbracket #1 \rrbracket}

\renewcommand\>{\rangle}
\newcommand\<{\langle}

\theoremstyle{plain}
\newtheorem{Theorem}[equation]{Theorem}
\newtheorem{Lemma}[equation]{Lemma}

\theoremstyle{remark}
\newtheorem{Remark}[equation]{Remark}
\theoremstyle{definition}

\newtheorem{Example}[equation]{Example}

\newtheorem{Problem}[equation]{Problem}

\allowdisplaybreaks
\begin{document}
\title[Factorization length distribution for affine semigroups I]{Factorization length distribution for affine \\ semigroups I: numerical semigroups with three generators}

\author{Stephan Ramon Garcia}
\address{Department of Mathematics, Pomona College, 610 N. College Ave., Claremont, CA 91711} 
\email{stephan.garcia@pomona.edu}
\urladdr{\url{http://pages.pomona.edu/~sg064747}}

\author{Christopher O'Neill}
\address{Mathematics Department\\University of California Davis\\Davis, CA 95616}
\email{coneill@math.ucdavis.edu}
\urladdr{https://www.math.ucdavis.edu/~coneill}

\author{Samuel Yih}
\address{Department of Mathematics, Pomona College, 610 N. College Ave., Claremont, CA 91711} 
\email{sy012014@mymail.pomona.edu}

\subjclass[2010]{Primary: 20M14, 05E40.}

\keywords{numerical semigroup; factorization; quasipolynomial; Erd\H{o}s--Straus conjecture; Egyptian fractions}

\thanks{First author partially supported by a David L. Hirsch III and Susan H. Hirsch Research Initiation Grant and the Institute for Pure and Applied Mathematics (IPAM)}

\begin{abstract}
Most factorization invariants in the literature extract extremal factorization behavior, such as the maximum and minimum factorization lengths.   Invariants of intermediate size, such as the mean, median, and mode factorization lengths are more subtle.  We use techniques from analysis and probability to describe the asymptotic behavior of these invariants.  Surprisingly, the asymptotic median factorization length is described by a number that is usually irrational.
\end{abstract}

\maketitle

\section{Introduction}
\label{sec:intro}

Let $\NN$ denote the set of nonnegative integers.  A \emph{numerical semigroup} is an additive subsemigroup $S \subset \NN$ (that is, a subset that is closed under addition).  Any numerical semigroup can be generated by finitely many integers $n_1, n_2,\ldots, n_k$, so we usually specify a numerical semigroup $S$ by writing
\begin{equation*}
S = \<n_1, \ldots, n_k\> = \{a_1n_1 + \cdots + a_kn_k : a_1, \ldots, a_k \in \NN\}.
\end{equation*}
In this paper, we assume $n_1 < \cdots < n_k$ and $S$ has finite complement in~$\NN$ (or, equivalently, that $\gcd(n_1, \ldots, n_k) = 1$), but we do \emph{not} assume $n_1, \ldots, n_k$ minimally generate~$S$.  For an introduction to numerical semigroups, we recommend~\cite{NSBook}.  

A \emph{factorization} of an element $n \in S$ is an expression 
\begin{equation*}
n = a_1n_1 + \cdots + a_kn_k
\end{equation*}
of $n$ as a sum of generators of $S$, which we often represent with the $k$-tuple $a = (a_1, a_2, \ldots, a_k) \in \NN^k$.  The set of all factorizations of $n$ is denoted $\mathsf Z_S(n) \subset \NN^k$. 
The \emph{length} of a factorization $a \in \mathsf Z_S(n)$ is the number $|a| = a_1 + \cdots + a_k$ of generators appearing in the sum, and the \emph{length set} of $n$ is
\begin{equation*}
\mathsf L_S(n) = \{|a| : a \in \mathsf Z_S(n)\}
\end{equation*}
of distinct factorization lengths of $n \in S$.  

Much of the factorization theory literature centers around \emph{factorization invariants}, which are discrete quantities used to classify and quantify the underlying factorization structure.  The study of factorization invariants is a thriving area; see~\cite{nonuniq,omegamonthly,factornumerical,factorizationtheoryproceedings} for an overview.  Two such invariants include $\MM(n) = \max \mathsf L(n)$ (the \emph{maximum factorization length} of~$n$), and $\mm(n) = \min \mathsf L(n)$ (the \emph{minimum factorization length} of~$n$).  The asymptotic behavior of these two invariants, studied in \cite{asymptoticlengthfunctions,minmaxquasi}, is characterized as follows.  

\begin{Theorem}[\!\!{\cite[Theorems~4.2 and~4.3]{minmaxquasi}}]\label{t:minmaxquasi}
If~$S = \<n_1, n_2, \ldots, n_k\>$, then for large $n \in S$,
\begin{equation}\label{eq:minmaxquasi}
\MM(n) = \tfrac{1}{n_1}n + c_0(n) 
\qquad\text{and}\qquad
\mm(n) = \tfrac{1}{n_k}n + c_0'(n),
\end{equation} 
where $c_0(n)$ and $c_0'(n)$ are rational-valued $n_1$- and $n_k$-periodic functions, respectively.
\end{Theorem}

One of the crowning achievements in factorization theory is the following \emph{structure theorem for sets of length}.  As a consequence, most invariants derived from factorization length focus on extremal lengths (e.g., $\MM(n)$ or $\mm(n)$) 
since this is where the ``interesting'' behavior occurs. 

\begin{Theorem}[\!\!{\cite[Theorem~4.3.6]{nonuniq}}]\label{t:lstructure}
Let $S = \<n_1, n_2, \ldots, n_k\>$ be a numerical semigroup.  There is an integer $M > 0$ such that for all $n \in S$, the length set $\mathsf L_S(n)$ equals an arithmetic sequence from which some subset of the first and last $M$ elements are removed.  
\end{Theorem}

In contrast, the goal of this paper is to initiate the study of factorizations of ``medium'' length.  To do so, we consider the \emph{length multiset} $\Le\multi{n}$ of $n$, by which we mean the multiset of factorization lengths in $\mathsf Z_S(n)$.
We focus specifically on the following factorization invariants.  

\begin{enumerate}[leftmargin=*]
\addtolength{\itemsep}{5pt}
\item 
The \emph{mean factorization length} $\mu(n)$ is the average factorization length of $n$:
\begin{equation*}
    \mu(n) = \frac{1}{|\Le \multi{n}|} \sum_{\ell \in \Le \multi{n}} \ell. 
\end{equation*}

\item 
The \emph{median factorization length} $\eta(n)$ is the median of $\Le \multi{n}$.

\item 
The \emph{mode frequency} $\nu(n)$ is the highest multiplicity among lengths in $\Le\multi{n}$.  Lengths with multiplicity $\nu(n)$ comprise the set $\gamma(n)$ of \emph{mode factorization lengths}.  

\end{enumerate}

Many results in the literature involving factorization invariants are asymptotic in nature \cite{FH06, GH92, CGGR02, asymptoticfactorizations}, in part because invariant behavior can be chaotic for ``small'' semigroup elements.  For instance, every multiset with elements from $\ZZ_{\ge 2}$ occurs as the length multiset of some squarefree numerical semigroup element \cite{lengthsetrealization}.  For some families of semigroups, including numerical semigroups, asymptotic results take the form of an \emph{eventually quasipolynomial} description like the one in Theorem~\ref{t:minmaxquasi} above \cite{OP14, GMV15, FIHF, minmaxquasi, dynamicalg, deltasetperiodic, catenarydegreeperiodic,0norm}.  

We seek asymptotic results for the mean, median, and mode factorization lengths of $n$ in the spirit of Theorem~\ref{t:minmaxquasi}.  Although we focus on $3$-generated numerical semigroups (as is common in the literature \cite{GLM17, GLM18, CGTV16, CPS14, AG10, RG04}), getting a hold of these intermediate factorization invariants is subtle enough to require several analytic and probabilistic techniques not standard in the factorization theory literature.  Moreover, we provide several examples which illustrate that the situation for four or more generators is substantially more complicated (Example~\ref{e:embdim4}).  

Central to the study of factorization invariants is the notion of a trade, which encodes a relation between semigroup generators.  More precisely, a \emph{trade} of a numerical semigroup $S$ is a pair $(z \mid z')$ of factorizations $z, z' \in \mathsf Z(n)$ for some element $n \in S$. A trade is \emph{length preserving} if $|z| = |z'|$.  

If $S = \<n_1 ,n_2 ,n_3\>$, then there is a unique trade of the form 
$$(a, 0, c \mid 0, b, 0)$$
with $a + c = b$ such that given a factorization of $n$ of length~$\ell$, all other factorizations of $n$ with the same length $\ell$ are obtained by repeatedly performing this trade~\cite[Theorem~1.3]{ENS}.  Phrased in the language of trades, we say $S$ has exactly one minimal length-preserving trade.  In this case, we refer to the element
\begin{equation*} 
t = an_1 + cn_3 = bn_2
\end{equation*}
as the \emph{trade element} of $S$.  Note that the same does not hold for numerical semigroups with more than three generators~\cite[Example~1.4]{ENS}.  

\begin{Example}\label{example:mcnugget}
The \emph{McNugget semigroup}~\cite{recreations} is given by $S = \<6, 9, 20\>$.  Its minimal length-preserving trade is $(11,0,3 \mid 0,14,0)$, making its trade element 
\begin{equation*}
t = 11 \cdot 6 + 3 \cdot 20 = 14 \cdot 9 = 126.
\end{equation*}
In~this semigroup, 
\begin{align*}
\mathsf Z(132) 
& = \{(2, 0, 6), (0, 8, 3), (3, 6, 3), (6, 4, 3), (9, 2, 3), (12, 0, 3), (1, 14, 0),\\
& \phantom{{}= \{} (4, 12, 0), (7, 10, 0), (10, 8, 0), (13, 6, 0), (16, 4, 0), (19, 2, 0), (22, 0, 0)\}
\end{align*}
and 
\begin{equation*}
\Le \multi{132} = \{\!\!\{8, 11, 12, 13, 14, 15, 15, 16, 17, 18, 19, 20, 21, 22\}\!\!\}.
\end{equation*}
As guaranteed by Theorem~\ref{t:lstructure}, $\Le_S(132)$ forms an arithmetic progression from $\mm(132) = 8$ to $\MM(132) = 22$ with step size $1$ and elements $9$ and $10$ omitted.  The mean factorization and median factorization lengths of $132$ are
\begin{equation*}
\mu(132) = 221/14 \approx 15.7857
\qquad\text{and}\qquad
\eta(n) =  \tfrac{1}{2}(15+16) = 31/2,
\end{equation*}
respectively.
Observe that $15$ occurs twice in $\Le \multi{132}$ because $132$ has two distinct factorizations of length $15$, namely $(1,14,0)$ and $(12,0,3)$.  No other factorization length appears this often,
so the set of mode factorization lengths is a singleton.  In particular,
\begin{equation*}
\gamma(132) = \{15\}
\qquad \text{and} \qquad
\nu(132) = 2.
\end{equation*}
On the other hand, $1001 \in S$ has $5$ factorizations of length $87$, namely
\begin{equation*}
(8, 57, 22), (19, 43, 25), (30, 29, 28), (41, 15, 31), (52, 1, 34) \in \mathsf Z(1001).
\end{equation*}
Each factorization is obtained from $(8, 57, 22)$ by repeatedly trading fourteen copies of $9$ for eleven copies of $6$ and three copies of $20$.  The mode frequency of $1001$ is $\nu(1001) = 8$, which is the number of factorizations of $1001$ of each length in 
\begin{equation*}
\gamma(1001) = \{107, 108, 110, 111, 112, 113, 114, 115\} \subset \mathsf L(1001).
\end{equation*}
\end{Example}

This paper is devoted to the proofs of three theorems, which describe the limiting behavior of the mode length and frequency (Theorem~\ref{t:mode}), the mean factorization length (Theorem~\ref{t:mean}), and the median factorization length (Theorem~\ref{t:median}).  We state and discuss each theorem here.

\begin{Theorem}[Mode length and frequency]\label{t:mode}
Fix a numerical semigroup $S = \<n_1, n_2, n_3\>$ and let $t$ be the trade element of $S$. 
For each $n \in S$, 
\begin{equation*}\label{eq:freqquasi}
\nu(n) = \tfrac{1}{t}n + c_0(n),
\end{equation*}
in which $c_0(n)$ is a rational-valued periodic function with period $t$.  Moreover, 
\begin{equation*}\label{eq:modequasi}
\gamma(n+t) = \gamma(n) + t/n_2 = \{m + t/n_2: m \in \gamma(n)\}
\end{equation*}
for all $n \in S$.  
\end{Theorem}

The proof of Theorem~\ref{t:mode} is contained in Section~\ref{sec:mode}.  Although its proof is a relatively straightforward combinatorial one, we leverage it and several analytic and probabilistic techniques to obtain the more technical results below.  

\begin{Theorem}[Mean factorization length]\label{t:mean}
For any numerical semigroup $S = \<n_1, n_2, n_3\>$,
\begin{equation}\label{eq:meanasymptotic}
\lim_{n \to \infty} \frac{\mu(n)}{n} = \frac{1}{3} \left( \frac{1}{n_1} + \frac{1}{n_2} + \frac{1}{n_3} \right ).
\end{equation}
\end{Theorem}

Theorem~\ref{t:mean} states that $\mu(n)$ grows asymptotically like $n$ times the reciprocal of the harmonic mean of the generators of $S$.  The largest value attained by~\eqref{eq:meanasymptotic} is $47/180 = 0.26\overline{1}$, which occurs for $S = \< 3,4,5\>$.
On the other hand, \eqref{eq:meanasymptotic} can be made arbitrarily small by selecting $n_1$ large enough.  This prompts the following.

\begin{Problem}\label{prob:erdos}
Which rational values in $(0,\frac{47}{180})$ are of the form~\eqref{eq:meanasymptotic}
for some numerical semigroup $S = \< n_1,n_2,n_3 \>$?
\end{Problem}

\begin{Remark}
An \emph{Egyptian fraction} is a finite sum of distinct rational numbers, each with numerator $1$ (historically the ancient Egyptians also permitted $2/3$ and $3/4$, for which they had special symbols).  Every positive rational number can be expressed as an Egyptian fraction, although these representations are not unique~\cite[Section~D11]{Guy}.  There are many rational numbers that cannot be expressed as Egyptian fractions with three or fewer terms.  For example,
\begin{equation*}
\frac{8}{11} = \frac{1}{2} + \frac{1}{5} + \frac{1}{37} + \frac{1}{4070}
\end{equation*}
has sixteen decompositions of length $4$, but none of length $3$.  Thus, there does not exist a numerical semigroup $S = \<n_1, n_2, n_3\>$ so that~\eqref{eq:meanasymptotic}
equals $\frac{8}{33} = 0.\overline{24}$.   There are other ``forbidden'' values of
\eqref{eq:meanasymptotic}, such as
\begin{equation*}
\frac{8}{51} \approx 0.156863 ,\qquad
\frac{3}{19} \approx 0.157895,\quad\text{and}\quad
\frac{14}{57} \approx 0.245614,
\end{equation*}
since none of $8/17$, $9/19$, and $14/19$ can be expressed as a sum of three unit fractions
\cite{Knott}. These sorts of problems are known to be extremely difficult.
The famed Erd\H{o}s--Straus conjecture 
asserts that for each $n \geq 2$,
\begin{equation*}
\frac{4}{n} = \frac{1}{n_1} + \frac{1}{n_2} + \frac{1}{n_3}
\end{equation*}
has a solution in nonnegative integers~\cite{Erdos, Tao, Tao2, Elsholtz, Graham, Vaughan}.
The numerator $4$ is replaced by $5$ in a closely-related conjecture of Sierpi\'nski~\cite{Sierpinski}.  Another
conjecture along these lines is due to Schinzel~\cite{Schinzel}: for each $k \in \NN$, the equation
\begin{equation*}
\frac{k}{n} = \frac{1}{n_1} + \frac{1}{n_2} + \frac{1}{n_3}
\end{equation*}
has a solution in nonnegative integers if $n$ is sufficiently large.  All of these conjectures
remain unresolved.  
\end{Remark}

In a striking departure from other factorization invariants, which typically coincide with rational-valued quasipolynomials for large semigroup elements, the expression~\eqref{eq:medianasymptotic} below is often an irrational number.  This reinforces our assertion that obtaining information about intermediate factorization invariants is both more difficult and more interesting than studying traditional extremal invariants.

\begin{Theorem}[Median factorization length]\label{t:median}
Suppose $S = \< n_1,n_2,n_3 \>$, and let
\begin{equation}\label{eq:F}
F = \frac{n_1(n_3-n_2)}{n_2(n_3-n_1)}.
\end{equation}
Then
\begin{equation}\label{eq:medianasymptotic}
\lim_{n\to\infty} \dfrac{\eta(n)}{n} =
\begin{cases}
\dfrac{1}{n_1}\left( 1 - \sqrt{ \dfrac{1-F}{2}} \right)  + \dfrac{1}{n_3}\sqrt{ \dfrac{1-F}{2}} 
& \text{if $F \leq \dfrac{1}{2}$}, \\[20pt]
\dfrac{1}{n_1} \sqrt{ \dfrac{F}{2}} + \dfrac{1}{n_3}\left( 1 - \sqrt{ \dfrac{F}{2}} \right)
& \text{if $F \geq \dfrac{1}{2}$}.
\end{cases}
\end{equation}%
\end{Theorem}

Regarding $n_1$ and $n_3$ as fixed, we see that the expressions in~\eqref{eq:medianasymptotic} are convex combinations of $1/n_1$ and $1/n_3$ with coefficients in terms of the constant $F$ (called the \emph{fulcrum constant}, see Section~\ref{sec:approx}).  We see from~\eqref{eq:F} that $F$
varies from $0$ to $1$ as $n_2$ varies from $n_3$ to $n_1$, and the coefficients of $1/n_1$ and $1/n_3$ vary between $1-1/\sqrt{2} \approx 0.29$ and $1/\sqrt{2} \approx 0.71$. Since the combinations are convex and $1/n_1 > 1/n_3$, it follows that 
\begin{equation*}
\lim_{n \to \infty} \frac{\eta(n)}{n} \in \left[
\left(\frac{2 - \sqrt{2}}{2} \right)\frac{1}{n_1} + 
\left(\frac{\sqrt{2}}{2}\right)\frac{1}{n_3}, \,\, 
\left(\frac{\sqrt{2}}{2}\right) \frac{1}{n_1} + 
\left(\frac{2 - \sqrt{2}}{2}\right) \frac{1}{n_3}
\right]
\end{equation*}
where the endpoints are attained at $F = 0$ and $F = 1$, respectively. In particular, these bounds are sharp since $F = 0$ precisely when $n_2 = n_3$, and $F = 1$ precisely when $n_1 = n_2$. Furthermore, 
\begin{equation*}
F = \frac{1}{2} \qquad \text{if and only if} \qquad \frac{1}{n_2} = \frac{1}{2}\left(\frac{1}{n_1} + \frac{1}{n_3} \right)
\end{equation*}
(that is, when $n_2$ is the harmonic mean of $n_1$ and $n_3$).  Moreover, our proof shows that this occurs if and only if
\begin{equation}\label{eq:everything2}
 \lim_{n \to \infty} \frac{\mu(n)}{n} = \lim_{n \to \infty} \frac{\eta(n)}{n} = \frac{1}{n_2},
\end{equation}
which is entirely analogous to the formulas~\eqref{eq:minmaxquasi} that describe
the asymptotic behavior of $\mm(n)$ and $\MM(n)$.  

In cases where~\eqref{eq:medianasymptotic} is rational, it is natural to ask whether the median function $\eta(n)$ is eventually quasilinear, in the spirit of Theorems~\ref{t:minmaxquasi} and~\ref{t:mode}.  Indeed, preliminary computations indicate the answer is ``yes'' for $S = \<12, 15, 20\>$ and ``no'' for $S = \<7, 16, 25\>$.  With this in mind, we pose the following question.  

\begin{Problem}\label{prob:medianquasi}
For which 3-generated numerical semigroups $S = \<n_1, n_2, n_3\>$ is the median function $\eta(n)$ eventually quasilinear?  
\end{Problem}

\section{Mode factorization lengths}
\label{sec:mode}

In this section, we prove Theorem~\ref{t:mode}.  

\begin{proof}[Proof of Theorem~\ref{t:mode}]
Fix $n \in S$ and $\ell \in \gamma(n)$, and suppose $(a,0,c \mid 0,b,0)$ is the unique minimal length-preserving trade in $S$.  If $n_1 = n_3$, then $t = n_1$ and $t \mid n$, so~$\nu(n) = n/t$.  Otherwise, since no element of the 2-generated semigroup $\<n_1, n_3\>$ has multiple factorizations of the same length \cite{NSBook}, there is a unique length-$\ell$ factorization $(a_1,a_2,a_3)$ of $n$ with maximal second coefficient.  The factorizations of $n$ with length $\ell$ can then be enumerated as
\begin{equation*} 
(a_1,a_2,a_3), \quad (a_1+a,\,a_2-b,\,a_3+c), \quad \ldots, \quad (a_1+ma,\,a_2-mb,\,a_3+mc),
\end{equation*} 
in which $m+1 = \nu(n)$. Each of these induces a factorization of $n+t$, namely
\begin{equation*} 
(a_1,a_2+b,a_3),\quad (a_1+a,\,a_2,\,a_3+c), \quad \ldots, \quad (a_1+ma,\,a_2-(m-1)b,\,a_3+mc) ,
\end{equation*} 
each with length $\ell + b$. Thus,
\begin{equation}\label{eq:modeinequality}
\nu(n+t) \ge \nu(n) + 1,
\end{equation}
with the last factorization of length $\ell + b$ given by 
\begin{equation*}
\big(a_1+(m+1)a,\,\, a_2-mb,\,\, a_3+(m+1)c \big).
\end{equation*}
Conversely, let $(b_1,b_2,b_3)$ be the unique factorization of $n+t$ with mode length $\ell' \in \gamma(n+t)$ and maximal second coefficient, and let $m' = \nu(n + t) - 1$.  As before, the $m' + 1$ factorizations 
\begin{equation*} 
(b_1, b_2, b_3), \quad (b_1 + a, \,\, b_2 - b, \,\, b_3 + c), \quad \ldots, \quad \big(b_1 + m'a, \,\, b_2 - m'b, \,\, b_3 + m'c\big) 
\end{equation*} 
of $n + t$ with length $\ell'$ induce $m'$ factorizations of $n$ with length $\ell' - b$, namely
\begin{align*} 
(b_1, b_2 - b, b_3),
&\quad (b_1 + a, \,\, b_2 - 2b, \,\, b_3 + c),\ldots, \\ 
&\qquad \big(b_1 + (m'-1)a, \,\, b_2 - m'b, \,\, b_3 + (m'-1)c \big).
\end{align*} 
Together with~\eqref{eq:modeinequality}, we conclude $\nu(n + t) = \nu(n) + 1$, which immediately yields
\begin{equation*}
\nu(n) = \tfrac{1}{t}n + c_0(n)
\end{equation*}
for some $t$-periodic function $c_0(n)$.  Lastly, the map between factorizations of $n$ and $n + t$ above yields a bijection 
\begin{equation*}
\gamma(n) \to \gamma(n + t)
\end{equation*}
that sends $\ell \mapsto \ell + b$.  This completes the proof.  
\end{proof}

\section{Continuous approximation of length multiplicities}
\label{sec:approx}

In this section, we develop machinery used in the proofs of Theorems~\ref{t:mean} and~\ref{t:median}. We begin by studying the values of $\mu$ and $\eta$ along certain subsequences of $S$, and then prove these values coincide asymptotically with those over the entire semigroup. Before doing so, we need to introduce one more quantity.  



Fix $n \in S = \<n_1, n_2, n_3\>$, and let 
\begin{equation*}
\delta = \gcd(n_3 - n_1, n_3 - n_2, n_2 - n_1).  
\end{equation*}
This constant arises in factorization theory as the minimum element of the \emph{delta set} of $S$; see \cite{DSB} for a full development.  Most relevant to our setting, $\delta$ is the step size of every arithmetic sequence in the structure theorem (Theorem~\ref{t:lstructure}).  

Theorem~\ref{t:deltabasic} demonstrates that $\delta$ is closely related to the trade element $t$ for 3-generated numerical semigroups.

\begin{Theorem}[{\cite[Theorem~1.3]{ENS}}]\label{t:deltabasic}
If $S = \<n_1, n_2, n_3\>$ is a numerical semigroup, then
\begin{equation*}
\delta t = n_2(n_3 - n_1),
\end{equation*}
where $t$ denotes the trade element of $S$.  
\end{Theorem}


Theorems~\ref{t:minmaxquasi},~\ref{t:mode}, and~\ref{t:deltabasic} demonstrate that the constants $\delta$, $t$, $n_1$, $n_2$, and $n_3$ each contain important information about the factorization structure of $S$, which suggests examining multiples of $s = \delta tn_1n_2n_3$.  Since the only factorization of $0 \in S$ is the length 0 factorization $(0,0,0)$, we obtain
\begin{align}\label{eq:akbk}
a_k &= \mm(ks) = \tfrac{1}{n_3}ks = k\delta tn_1n_2 \quad \text{and} \\
b_k &= \MM(ks) = \tfrac{1}{n_1}ks = k\delta tn_2n_3
\end{align}
from Theorem~\ref{t:minmaxquasi} for each $k \in \NN$.  Moreover, $\gamma(ks) = \{c_k\}$ is singleton and precisely $d_k$ minimal trades can be made from the factorization $(0, c_k, 0)$ of $ks$, where
\begin{align*}
c_k &= \tfrac{1}{n_2}ks = k\delta tn_1n_3 \quad \text{and} \\
d_k &= \nu(ks) - 1 = \tfrac{1}{t}ks = k\delta n_1n_2n_3
\end{align*}
from Theorem~\ref{t:mode}.  In what follows, let $f_k(x)$ denote the multiplicity of the length $x \in \Le(sk)$.  Figure~\ref{f:multiplicities} depicts the multiplicity function of $630 \in \<3,5,7\>$.  

\begin{figure}
\includegraphics[width=0.75\textwidth]{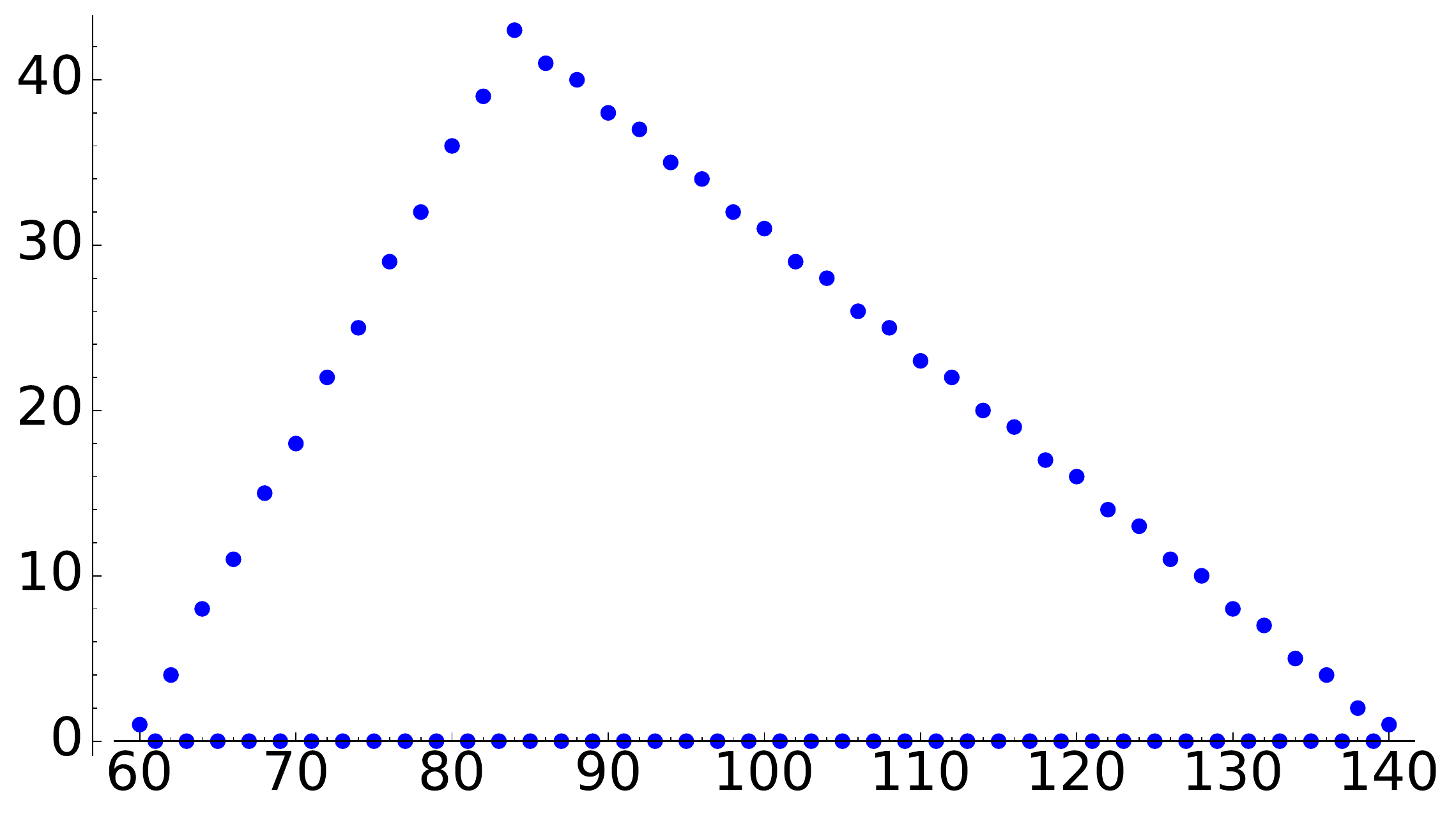}
\caption{Factorization length multiplicity function of $630 \in \<3, 5, 7\>$ plotted against factorization length.  In this semigroup, $\delta = 2$, so the function assumes the value 0 on every other length.  Disregarding the frequent dips to zero, the function increases roughly linearly from the minimum factorization length $(90)$ to the mode length $(126)$, then decreases roughly linearly from mode length to maximum factorization length $(210)$.}
\label{f:multiplicities}
\end{figure}

To characterize the asymptotic behavior of $f_k$, we construct a step function $S_k$ with the property that
\begin{equation*}
\sum_{i=a_k}^{\ell-1} f_k(i) = \frac{1}{\delta} \int_{a_k}^\ell S_k(x) \,dx
\end{equation*}
for all $\ell \in \Le(ks)$.  We give $S_k$ one step of width $\delta$ for each length $\ell \in \Le(ks)$.  
Steps corresponding to $\ell \in [a_k, c_k)$ agree with $f_k(\ell)$ on their left endpoints, leaving their right endpoints open,
and steps corresponding to $\ell \in (c_k, b_k]$ agree with $f_k(\ell)$ at their right endpoints, leaving their left endpoints open.  
The value $S_k(c_k)$ is left undefined, which is not a problem since we intend to integrate $S_k$.  Formally, this yields
\begin{equation*}
S_k(x) = 
\begin{cases}
f_k(a_k + j \delta) & \text{if $x < c_k$, $x \in \big[a_k + j \delta,\,a_k + (j+1) \delta \big)$}, \\[5pt]
f_k(c_k + j \delta) & \text{if $x > c_k$, $x \in \big(c_k + (j-1) \delta ,\,c_k + j \delta\big]$}.
\end{cases}
\end{equation*}
See Figure~\ref{f:stepmultiplicities} for a depiction of the step function constructed from Figure~\ref{f:multiplicities}.  

\begin{figure}
\includegraphics[width=0.75\textwidth]{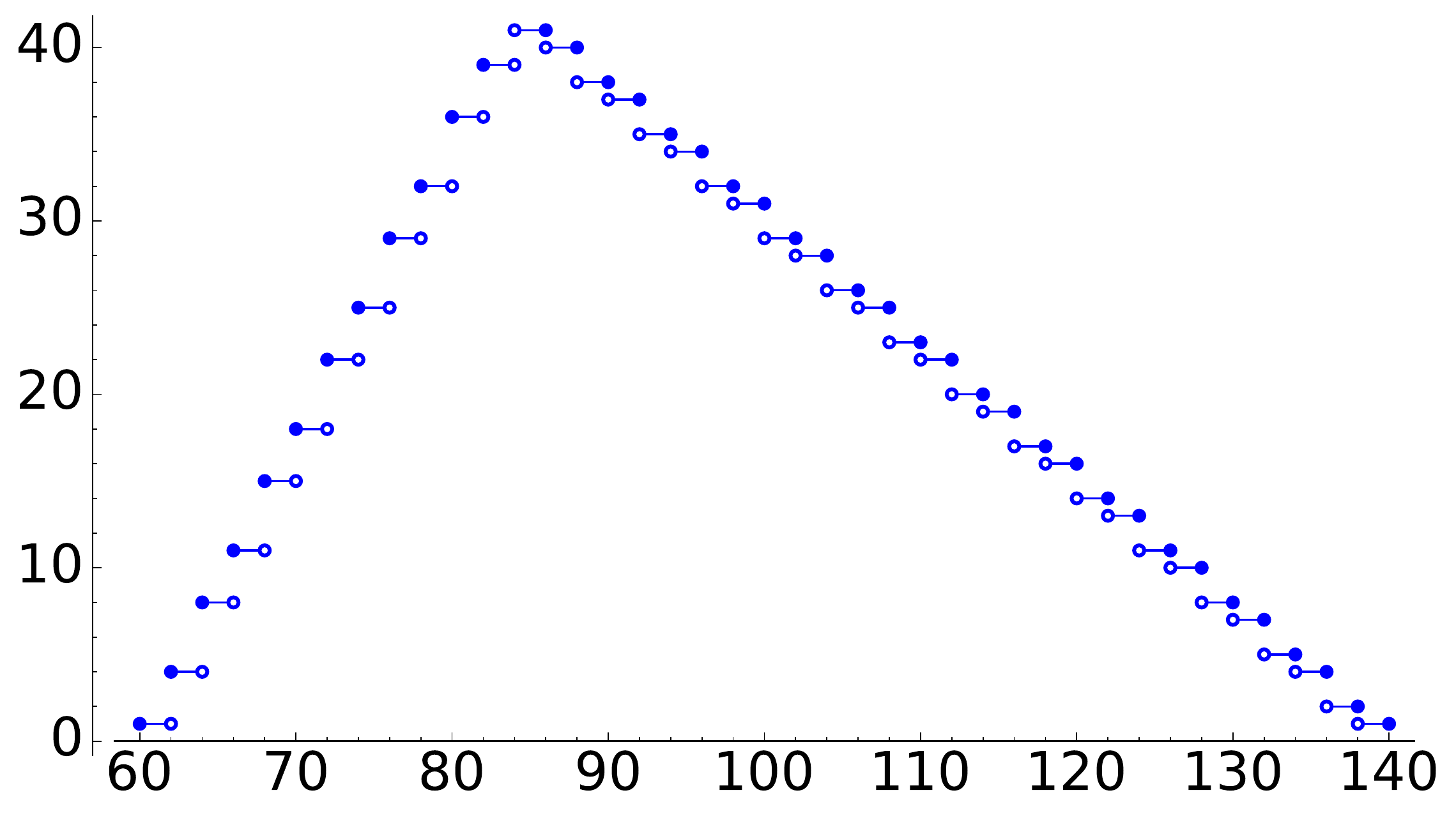}
\caption{
The step function corresponding to the same semigroup element as Figure~\ref{f:multiplicities}. 
} 
\label{f:stepmultiplicities}
\end{figure}

In order to describe the asymptotic behavior of $S_k$ as $k \to \infty$, we prove that, after appropriate affine transformations, $S_k$ converges uniformly to a piecewise linear function $L_k$ on $[a_k, b_k]$ connecting the points $(a_k, 1)$, $(c_k, d_k + 1)$, and $(b_k, 1)$.  More precisely, $L_k$ is given by
\begin{equation*}
L_k(x) = 
\begin{cases}
\dfrac{n_2n_3}{t(n_3 - n_2)}(x - a_k) + 1 & \text{ if } x \in [a_k, c_k] ,\\[10pt]
\dfrac{n_1n_2}{t(n_2 - n_1)}(b_k - x) + 1 & \text{ if } x \in [c_k, b_k].
\end{cases}
\end{equation*}
This is depicted in Figure~\ref{f:upperlowerbounds} on the same axes as $S_k$.  

We begin by showing that $L_k$ is an upper bound for the step function $S_k$.  

\begin{Lemma}\label{l:upperbound}
For each integer $x \in [a_k, b_k]$, the multiplicity $f_k(x)$ is at most $L_k(x)$.  
In particular $S_k(x) \le L_k(x)$ for all $x \in [a_k, b_k]$.  
\end{Lemma}

\begin{proof}
First, suppose $x = a_k + \ell \in [a_k, c_k]$, so that
\begin{equation*} 
L_k(x) = \dfrac{n_2n_3}{t(n_3 - n_2)}\ell + 1.
\end{equation*} 
If no factorizations of length $x$ exist, then we are done.  Otherwise, it suffices to prove that any length $x$ factorization $z = (z_1,z_2,z_3)$ must satisfy
$$z_2 \le \frac{n_3}{n_3 - n_2}\ell,$$
as Theorem~\ref{t:deltabasic} then bounds the number of minimal length-preserving trades that can be applied to $z$.  Indeed, since $ks = a_kn_3 = z_1n_1 + z_2n_2 + z_3n_3$, we see
$$n_3\ell = n_3(z_1 + z_2 + z_3) - n_3a_k = z_1(n_3 - n_1) + z_2(n_3 - n_2) \ge z_2(n_3 - n_2),$$
from which we draw the desired conclusion.  

Since the case $x \in [c_k, b_k]$ is analogous to the first, we omit its proof.  
\end{proof}

We next exhibit a lower bound for $S_k$ (also depicted in Figure~\ref{f:upperlowerbounds}) that we will see also converges uniformly to $L_k$ as $k \to \infty$ (after appropriate affine transfomations).  

\begin{Lemma}\label{l:lowerbound}
There is a constant $C$ independent of $k$ with the following properties.  
\begin{enumerate}[leftmargin=*]
\addtolength{\itemsep}{5pt}
\item 
For any $x \in [a_k, c_k)$, writing $x = a_k + t(n_3 - n_2)j + \ell$ with $0 \le \ell < t(n_3 - n_2)$ and $0 \le j < k\delta n_1$, we have
$f_k(x) \geq n_2n_3(j - C) + 1$.  

\item 
For any $x \in (c_k, b_k]$, writing $x = b_k - t(n_2 - n_1)j - \ell$ with $0 \le \ell < t(n_2 - n_1)$ and $0 \le j < k\delta n_3$, we have
$f_k(x) \geq n_1n_2(j - C) + 1$.  

\end{enumerate}
In particular, $|L_k - S_k|$ is bounded independent of $k$.  
\end{Lemma}


\begin{proof}
Suppose $x < c_k$, and write $x = a_k + \ell + t(n_3 - n_2)j$ with $0 \le \ell < t(n_3 - n_2)$.  By Theorem~\ref{t:lstructure}, $C$ can be chosen large enough to ensure $f_k(x)$ is positive for $j = C$.  For $j < C$, the statement is trivial since $f_k(x)$ is non-negative.  

Now, for $j > C$, we proceed by induction on $j$.  If $f_k(x) \ge n_2n_3(j - C) + 1$, then by Theorem~\ref{t:deltabasic} there is some factorization $z = (z_1, z_2, z_3)$ of length $x$ with second coordinate at least $z_2 \ge tn_3(j - C)$.  By Lemma~\ref{l:upperbound}, we must have 
\begin{align*}
a_k(n_3 - n_1) 
&= z_1n_1 + z_2n_2 + z_3n_3 - a_kn_1 = (x - z_2 - z_3)n_1 + z_2n_2 + z_3n_3 - a_kn_1 \\
&= (\ell + t(n_3 - n_2)j)n_1 + z_2(n_2 - n_1) + z_3(n_3 - n_1) \\
&\le (\ell + t(n_3 - n_2)j)n_1 + \frac{n_3}{n_3 - n_2}(\ell + t(n_3 - n_2)j)(n_2 - n_1) + z_3(n_3 - n_1) \\
&= \frac{n_2(n_3 - n_1)}{n_3 - n_2}\ell + tn_2(n_3 - n_1)j + z_3(n_3 - n_1),
\end{align*}
which, by assumption on $\ell$ and $j$, yields
\begin{equation*}
a_k \le \frac{n_2}{n_3 - n_2}\ell + tn_2j + z_3 < tn_2 + tn_2j + z_3 < tn_2 + k\delta tn_1n_2 + z_3 = tn_2 + a_k + z_3.
\end{equation*}
This means $(z_1, z_2 + tn_3, z_3 - tn_2)$ is a factorization of $ks$ with length $x + t(n_3 - n_2)$ and second coordinate at least $tn_3(j - C) + tn_3$, so we conclude
\begin{equation*}
f_k(x + t(n_3 - n_2)) \ge n_2n_3(j - C) + n_2n_3 + 1
\end{equation*}
by Theorem~\ref{t:deltabasic}.  

The case $x > c_k$ is similar.  Since
\begin{equation*} 
0 \leq L_k(x) - S_k(x) \leq Cn_1n_2n_3
\end{equation*}
for all $x \in [a_k, b_k]$ by Lemma~\ref{l:upperbound}, the final claim holds as well.  
\end{proof}

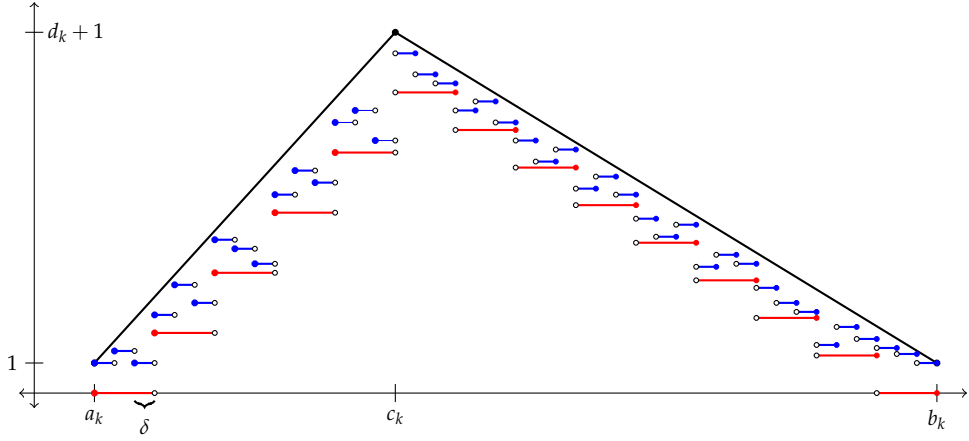
\begin{figure}
    \centering
    \begin{tikzpicture}[thick,scale=0.8, every node/.style={scale=0.8}]
        \draw[<->,thin](-.25,0)--(15.5,0);
        \draw[<->,thin](0,-.25)--(0,6.5);
      
        \draw[thick,black](1,0.5)to(6,6);
        \draw[thick,black](6,6)to(15,0.5);
        \draw[fill=black](1,0.5) circle (.04cm);
        \draw[fill=black](6,6) circle (.04cm);
        \draw[fill=black](15,0.5) circle (.04cm);

        \draw[thin] (-0.15,0.5)node[left]{$1$}--(0.15,0.5);
        \draw[thin] (-0.15,6)node[right, xshift=0.25cm]{$d_k+1$}--(0.15,6);

        \draw[thin] (1,-0.15)node[below]{$a_k$}--(1,0.15);
        \draw[thin] (6,-0.15)node[below]{$c_k$}--(6,0.15);
        \draw[thin] (15,-0.15)node[below]{$b_k$}--(15,0.15);

        \draw[thick,blue](1.0,0.5)to(1.3333,0.5);  \draw[fill=blue,blue](1.0,0.5) circle (0.04cm);   \draw[fill=white,thin](1.3333,0.5) circle (0.04cm);
        \draw[thick,blue](1.3333,0.7)to(1.6667,0.7);  \draw[fill=blue,blue](1.3333,0.7) circle (0.04cm);   \draw[fill=white,thin](1.6667,0.7) circle (0.04cm);
        \draw[thick,blue](1.6667,0.5)to(2,0.5);  \draw[fill=blue,blue](1.6667,0.5) circle (0.04cm);   \draw[fill=white,thin](2,0.5) circle (0.04cm);

        \draw[thick,blue](2.0,1.3)to(2.3333,1.3);  \draw[fill=blue,blue](2.0,1.3) circle (0.04cm);   \draw[fill=white,thin](2.3333,1.3) circle (0.04cm);
        \draw[thick,blue](2.3333,1.8)to(2.6667,1.8);  \draw[fill=blue,blue](2.3333,1.8) circle (0.04cm);   \draw[fill=white,thin](2.6667,1.8) circle (0.04cm);
        \draw[thick,blue](2.6667,1.5)to(3.0,1.5);  \draw[fill=blue,blue](2.6667,1.5) circle (0.04cm);   \draw[fill=white,thin](3.0,1.5) circle (0.04cm);

        \draw[thick,blue](3.0,2.55)to(3.3333,2.55);  \draw[fill=blue,blue](3.0,2.55) circle (0.04cm);   \draw[fill=white,thin](3.3333,2.55) circle (0.04cm);
        \draw[thick,blue](3.3333,2.4)to(3.6667,2.4);  \draw[fill=blue,blue](3.3333,2.4) circle (0.04cm);   \draw[fill=white,thin](3.6667,2.4) circle (0.04cm);
        \draw[thick,blue](3.6667,2.15)to(4.0,2.15);  \draw[fill=blue,blue](3.6667,2.15) circle (0.04cm);   \draw[fill=white,thin](4.0,2.15) circle (0.04cm);
       
        \draw[thick,blue](4.0,3.3)to(4.3333,3.3);  \draw[fill=blue,blue](4.0,3.3) circle (0.04cm);   \draw[fill=white,thin](4.3333,3.3) circle (0.04cm);
        \draw[thick,blue](4.3333,3.7)to(4.6667,3.7);  \draw[fill=blue,blue](4.3333,3.7) circle (0.04cm);   \draw[fill=white,thin](4.6667,3.7) circle (0.04cm);
        \draw[thick,blue](4.6667,3.5)to(5.0,3.5);  \draw[fill=blue,blue](4.6667,3.5) circle (0.04cm);   \draw[fill=white,thin](5.0,3.5) circle (0.04cm);
        
        \draw[thin,blue](5.0,4.5)to(5.3333,4.5);  \draw[fill=blue,blue](5.0,4.5) circle (0.04cm);   \draw[fill=white,thin](5.3333,4.5) circle (0.04cm);
        \draw[thin,blue](5.3333,4.7)to(5.6667,4.7);  \draw[fill=blue,blue](5.3333,4.7) circle (0.04cm);   \draw[fill=white,thin](5.6667,4.7) circle (0.04cm);
        \draw[thin,blue](5.6667,4.2)to(6.0,4.2);  \draw[fill=blue,blue](5.6667,4.2) circle (0.04cm);   \draw[fill=white,thin](6.0,4.2) circle (0.04cm);

        \draw[thick,blue](6.0,5.65)to(6.3333,5.65);  \draw[fill=white,thin](6.0,5.65) circle (0.04cm);   \draw[fill=blue,blue,thin](6.3333,5.65) circle (0.04cm);
        \draw[thick,blue](6.3333,5.3)to(6.6667,5.3);  \draw[fill=white,thin](6.3333,5.3) circle (0.04cm);   \draw[fill=blue,blue,thin](6.6667,5.3) circle (0.04cm);
        \draw[thick,blue](6.6667,5.15)to(7.0,5.15);  \draw[fill=white,thin](6.6667,5.15) circle (0.04cm);   \draw[fill=blue,blue,thin](7.0,5.15) circle (0.04cm);

        \draw[thick,blue](7.0,4.7)to(7.3333,4.7);  \draw[fill=white,thin](7.0,4.7) circle (0.04cm);   \draw[fill=blue,blue,thin](7.3333,4.7) circle (0.04cm);
        \draw[thick,blue](7.3333,4.85)to(7.6667,4.85);  \draw[fill=white,thin](7.3333,4.85) circle (0.04cm);   \draw[fill=blue,blue,thin](7.6667,4.85) circle (0.04cm);
        \draw[thick,blue](7.6667,4.5)to(8.0,4.5);  \draw[fill=white,thin](7.6667,4.5) circle (0.04cm);   \draw[fill=blue,blue,thin](8.0,4.5) circle (0.04cm);

        \draw[thick,blue](8.0,4.2)to(8.3333,4.2);  \draw[fill=white,thin](8.0,4.2) circle (0.04cm);   \draw[fill=blue,blue,thin](8.3333,4.2) circle (0.04cm);
        \draw[thick,blue](8.3333,3.85)to(8.6667,3.85);  \draw[fill=white,thin](8.3333,3.85) circle (0.04cm);   \draw[fill=blue,blue,thin](8.6667,3.85) circle (0.04cm);
        \draw[thick,blue](8.6667,4.05)to(9.0,4.05);  \draw[fill=white,thin](8.6667,4.05) circle (0.04cm);   \draw[fill=blue,blue,thin](9.0,4.05) circle (0.04cm);

        \draw[thick,blue](9.0,3.4)to(9.3333,3.4);  \draw[fill=white,thin](9.0,3.4) circle (0.04cm);   \draw[fill=blue,blue,thin](9.3333,3.4) circle (0.04cm);
        \draw[thick,blue](9.3333,3.6)to(9.6667,3.6);  \draw[fill=white,thin](9.3333,3.6) circle (0.04cm);   \draw[fill=blue,blue,thin](9.6667,3.6) circle (0.04cm);
        \draw[thick,blue](9.6667, 3.3)to(10.0,3.3);  \draw[fill=white,thin](9.6667,3.3) circle (0.04cm);   \draw[fill=blue,blue,thin](10.0, 3.3) circle (0.04cm);

        \draw[thick,blue](10.0,2.9)to(10.3333,2.9);  \draw[fill=white,thin](10.0,2.9) circle (0.04cm);   \draw[fill=blue,blue,thin](10.3333,2.9) circle (0.04cm);
        \draw[thick,blue](10.3333,2.6)to(10.6667,2.6);  \draw[fill=white,thin](10.3333,2.6) circle (0.04cm);   \draw[fill=blue,blue,thin](10.6667,2.6) circle (0.04cm);
        \draw[thick,blue](10.6667, 2.8)to(11.0,2.8);  \draw[fill=white,thin](10.6667,2.8) circle (0.04cm);   \draw[fill=blue,blue,thin](11.0, 2.8) circle (0.04cm);
        
        \draw[thick,blue](11.0,2.1)to(11.3333,2.1);  \draw[fill=white,thin](11.0,2.1) circle (0.04cm);   \draw[fill=blue,blue,thin](11.3333,2.1) circle (0.04cm);
        \draw[thick,blue](11.3333,2.3)to(11.6667,2.3);  \draw[fill=white,thin](11.3333,2.3) circle (0.04cm);   \draw[fill=blue,blue,thin](11.6667,2.3) circle (0.04cm);
        \draw[thick,blue](11.6667, 2.15)to(12.0,2.15);  \draw[fill=white,thin](11.6667,2.15) circle (0.04cm);   \draw[fill=blue,blue,thin](12.0, 2.15) circle (0.04cm);        
        
        \draw[thick,blue](12.0,1.75)to(12.3333,1.75);  \draw[fill=white,thin](12.0,1.75) circle (0.04cm);   \draw[fill=blue,blue,thin](12.3333,1.75) circle (0.04cm);
        \draw[thick,blue](12.3333,1.5)to(12.6667,1.5);  \draw[fill=white,thin](12.3333,1.5) circle (0.04cm);   \draw[fill=blue,blue,thin](12.6667,1.5) circle (0.04cm);
        \draw[thick,blue](12.6667, 1.35)to(13.0,1.35);  \draw[fill=white,thin](12.6667, 1.35) circle (0.04cm);   \draw[fill=blue,blue,thin](13.0, 1.35) circle (0.04cm);              

        \draw[thick,blue](13.0,0.8)to(13.3333,0.8);  \draw[fill=white,thin](13.0,0.8) circle (0.04cm);   \draw[fill=blue,blue,thin](13.3333,0.8) circle (0.04cm);
        \draw[thick,blue](13.3333,1.1)to(13.6667,1.1);  \draw[fill=white,thin](13.3333,1.1) circle (0.04cm);   \draw[fill=blue,blue,thin](13.6667,1.1) circle (0.04cm);
        \draw[thick,blue](13.6667, 0.9)to(14.0,0.9);  \draw[fill=white,thin](13.6667, 0.9) circle (0.04cm);   \draw[fill=blue,blue,thin](14.0, 0.9) circle (0.04cm);     
        
        \draw[thick,blue](14.0,0.75)to(14.3333,0.75);  \draw[fill=white,thin](14.0,0.75) circle (0.04cm);   \draw[fill=blue,blue,thin](14.3333,0.75) circle (0.04cm);
        \draw[thick,blue](14.3333,0.65)to(14.6667,0.65);  \draw[fill=white,thin](14.3333,0.65) circle (0.04cm);   \draw[fill=blue,blue,thin](14.6667,0.65) circle (0.04cm);
        \draw[thick,blue](14.6667, 0.5)to(15.0,0.5);  \draw[fill=white,thin](14.6667, 0.5) circle (0.04cm);   \draw[fill=blue,blue,thin](15.0, 0.5) circle (0.04cm);             

        \draw[thick,red](1,0)to(2,0);  \draw[fill=red,red](1,0) circle (0.04cm);   \draw[fill=white,thin](2,0) circle (0.04cm);
        \draw[thick,red](2,1)to(3,1);  \draw[fill=red,red](2,1) circle (0.04cm);   \draw[fill=white,thin](3,1) circle (0.04cm);
        \draw[thick,red](3,2)to(4,2); \draw[fill=red,red](3,2) circle (0.04cm);   \draw[fill=white,thin](4,2) circle (0.04cm);
        \draw[thick,red](4,3)to(5,3);  \draw[fill=red,red](4,3) circle (0.04cm);   \draw[fill=white,thin](5,3) circle (0.04cm);
        \draw[thick,red](5,4)to(6,4); \draw[fill=red,red](5,4) circle (0.04cm);   \draw[fill=white,thin](6,4) circle (0.04cm);
        \draw[thick,red](6,5)to(7,5);  \draw[fill=white,thin](6,5) circle (0.04cm);   \draw[fill=red,red,thin](7,5) circle (0.04cm);
        \draw[thick,red](7,4.375)to(8,4.375);  \draw[fill=white,thin](7,4.375) circle (0.04cm);   \draw[fill=red,red,thin](8,4.375) circle (0.04cm);
        \draw[thick,red](8,3.75)to(9,3.75);  \draw[fill=white,thin](8,3.75) circle (0.04cm);   \draw[fill=red,red,thin](9,3.75) circle (0.04cm);
        \draw[thick,red](9,3.1253)to(10,3.1253);  \draw[fill=white,thin](9,3.1253) circle (0.04cm);   \draw[fill=red,red,thin](10,3.1253) circle (0.04cm);
        \draw[thick,red](10,2.5)to(11,2.5);   \draw[fill=white,thin](10,2.5) circle (0.04cm);   \draw[fill=red,red,thin](11,2.5) circle (0.04cm);
        \draw[thick,red](11,1.875)to(12,1.875);  \draw[fill=white,thin](11,1.875) circle (0.04cm);   \draw[fill=red,red,thin](12,1.875) circle (0.04cm);
        \draw[thick,red](12,1.25)to(13,1.25);  \draw[fill=white,thin](12,1.25) circle (0.04cm);   \draw[fill=red,red,thin](13,1.25) circle (0.04cm);
        \draw[thick,red](13,0.625)to(14,0.625);  \draw[fill=white,thin](13,0.625) circle (0.04cm);   \draw[fill=red,red,thin](14,0.625) circle (0.04cm);
        \draw[thick,red](14,0)to(15,0);   \draw[fill=white,thin](14,0) circle (0.04cm);   \draw[fill=red,red,thin](15,0) circle (0.04cm);
        
        \draw [
            thick,
            decoration={
                brace,
                mirror,
                raise=0.1cm
            },
            decorate
        ] (1.6667,0)--(2,0) node [pos=0.5,anchor=north,yshift=-0.3cm] {$\delta$}; 
    \end{tikzpicture}
    \caption{A depiction of the upper bound $L_k$ (black diagonal lines) and lower bound (longer red steps) for $f_k$ (shorter blue steps).}
    \label{f:upperlowerbounds}
\end{figure}

Now that we have obtained upper and lower bounds for $S_k$, we normalize $L_k$ so it is a probability distribution  over $[0,1]$ by applying the invertible affine transformations $T_k: [a_k,b_k] \to [0,1]$ given by
\begin{equation*}
T_k(x) = \frac{x-a_k}{b_k-a_k}
\end{equation*}
and dividing by the normalization factor 
\begin{equation*}
\frac{\delta |\mathsf Z_S(ks)|}{b_k - a_k}
\end{equation*}
so the resulting function is nonnegative and integrates to $1$ over $[0,1]$.
The image under each $T_k$ of the mode factorization length~$c_k$ is the \emph{fulcrum constant}
\begin{equation}\label{eq:fulcrum}
F = \frac{n_1(n_3-n_2)}{n_2(n_3-n_1)}.
\end{equation}

In order to make things more precise, we require Big-$O$ and Big-$\Theta$ notation.  For two real-valued functions $f$ and $g$, we say that $f(x) = O(g(x))$ if there is a constant $C > 0$ such that $|f(x)| \le O(g(x))$ for sufficiently large $x$.  We say that $f(x) = \Theta(g(x))$ if there are constants $C_1,C_2$ so that 
$C_1|g(x)| \leq |f(x)| \leq C_2|g(x)|$
for sufficiently large $x$.

The linear component of $L_k$ on $[a_k, c_k]$ is mapped to the line
\begin{equation*} 
\frac{2}{F} x + \frac{b_k-a_k}{\delta|\mathsf Z_S(ks)|} = \frac{2}{F}x + \Theta(1/k)
\end{equation*}
on $[0,F]$ since $|\mathsf Z_S(ks)| = \Theta(k^2)$ by \cite[Theorem~3.9]{FIHF} and $b_k - a_k = \Theta(k)$ by~\eqref{eq:akbk}.  An analogous computation on the remaining linear component demonstrates that the normalized $L_k$ converges uniformly to the piecewise linear function
\begin{equation*}
L(x) = 
\begin{cases}
\dfrac{2}{F}x & \text{if $0 \leq x \leq F$}, \\
- \dfrac{2}{1-F}x + \dfrac{2}{1-F} & \text{if $F < x \leq 1$},
\end{cases} 
\end{equation*}
depicted in Figure~\ref{Figure:TikzL}, a standard triangular distribution over the interval $[0,1]$.  The properties of a triangular distribution are well known \cite{TRBook}; in this case,
\begin{equation}\label{eq:trimean}
\mu_L = \int_0^1 xL(x) \,dx = \frac{1+F}{3}
\end{equation}
and 
\begin{equation}\label{eq:trimed}
\eta_L = \begin{cases}
1 - \sqrt{\dfrac{1-F}{2}} & \text{ if } F \leq \frac{1}{2}, \\[10pt]
\sqrt{\dfrac{F}{2}} & \text{ if } F \geq \frac{1}{2}.
\end{cases} 
\end{equation}
The normalized $S_k$ are bounded above by the normalized $L_k$ by Lemma~\ref{l:upperbound}, and, as $|L_k - S_k|$ is bounded independent of $k$ by Lemma~\ref{l:lowerbound}, the difference between the normalized $S_k$ and $L_k$ tends to $0$. In particular, the normalized $S_k$ converge uniformly to $L$.  

\begin{figure}
    \centering
    \begin{tikzpicture}[thick,scale=0.5, every node/.style={scale=0.8}]
        \draw[->,thin](1,0)--(15.5,0);
        \draw[->,thin](1,0)--(1,5.5);
      
        \draw[thick,black](1,0)to(6,5);
        \draw[thick,black](6,5)to(15,0);
        \draw[fill=black](1,0) circle (.05cm);
        \draw[fill=black](6,5) circle (.05cm);
        \draw[fill=black](15,0) circle (.05cm);

        \draw[thin] (1,0)node[left, yshift=-0.25cm]{$0$}--(1.15,0);
        \draw[thin] (0.85,2.5)node[left]{$1$}--(1.15,2.5);
        \draw[thin] (0.85,5)node[left]{$2$}--(1.15,5);

        \draw[thin] (6,-0.15)node[below]{$F$}--(6,0.15);
        \draw[thin,dashed] (6,0)--(6,5);
        
        \draw[thin] (15,-0.15)node[below]{$1$}--(15,0.15);
    \end{tikzpicture}
    \caption{The function $L(x)$.}
    \label{Figure:TikzL}
\end{figure}
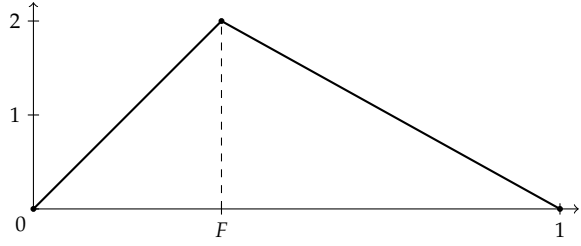

\section{Proving Theorems~\ref{t:mean} and~\ref{t:median}}
\label{sec:lastproofs}

In Section~\ref{sec:approx}, we obtained an asymptotic description of the length multiset $\Le\multi{n}$ for elements of the form $n = ks = k \delta t n_1 n_2 n_3$.  We begin by proving that the convergence of $\mu(n)/n$ and $\eta(n)/n$ for this subsequence implies convergence of the whole sequence (Lemma~\ref{l:subseqconv}).  Once this is done, we prove Theorems~\ref{t:mean} and~\ref{t:median}.  


\begin{Lemma}\label{l:subseqconv}
If $S = \<n_1, n_2, n_3\>$ is a numerical semigroup and $s = \delta tn_1n_2n_3$, then
\begin{equation*} 
L = \lim_{n \to \infty} \frac{\mu(n)}{n} \text{ exists }
\quad \text{if and only if} \quad
L' = \lim_{k \to \infty} \frac{\mu(sk)}{sk} \text{ exists. }
\end{equation*} 
If either limit exists (equivalently, both limits exist), then they converge to the same limit.  The same holds with~$\eta$ in place of~$\mu$.  
\end{Lemma}

\begin{proof} 
If $L$ exists, then every subsequence converges to the same limit, so $L' = L$.  Conversely, suppose $L'$ exists.  Since $S$ has finite complement, every $x \in S$ can be written as $x = ks + r$ for some integer $k \ge 0$ and $r \in S$ by \cite[Lemma~1.6]{NSBook}.  In~order to prove $L = L'$, it suffices to show that for fixed $r \in S$, 
\begin{equation}\label{eq:limcomp}
\lim_{k \to \infty} \frac{\mu(ks + r)}{ks + r} = \lim_{k \to \infty} \frac{\mu(ks)}{ks}.
\end{equation}
Given any $a \in \mathsf L(r)$, there is an injective map
\begin{equation*}
\begin{array}{rcl}
\Le\multi{ks} &\hookrightarrow& \Le\multi{ks + r} \\
\ell &\mapsto& \ell + a.
\end{array}
\end{equation*}
Write $\Le\multi{ks + r} = \{\!\!\{\ell_1, \ldots, \ell_M\}\!\!\}$ with $\ell_1, \ldots, \ell_m$ lying in the image of the above map and $\ell_{m+1}, \ldots, \ell_{M}$ lying outside the image; in particular, $M = |\mathsf Z(ks + r)|$ and $m = |\mathsf Z(ks)|$.  Using the facts that 
\begin{equation}\label{eq:Mmtheta}
|\mathsf Z(n)| = \Theta(n^2) \qquad \text{ and } \qquad |\mathsf Z(n + r)| - |\mathsf Z(n)| = \Theta(n)
\end{equation}
for fixed $r$ by \cite[Theorem~3.9]{FIHF} and that each element of $\Le\multi{ks}$ is $O(k)$ by Theorem~\ref{t:minmaxquasi}, we obtain $M = \Theta(k^2)$, $m = \Theta(k^2)$, $M - m = \Theta(k)$, and
\begin{align*}
|\mu(ks + r) - \mu(ks)|
&= \left|  \frac{1}{M} \sum_{i=1}^{M} \ell_i - \frac{1}{m} \sum_{i=1}^{m} (\ell_i - a) \right| \\
&= \left| \left(\frac{1}{M} - \frac{1}{m}\right) \sum_{i=1}^{m} \ell_i + \frac{1}{M} \sum_{i=m+1}^{M} \ell_i + a \right| \\
&\le \left| \left(\frac{1}{M} - \frac{1}{m}\right) \sum_{i=1}^{m} \ell_i \right| + \frac{1}{M} \sum_{i=m+1}^{M} \ell_i + a \\
&= O\big(\tfrac{1}{k^3}\big)O(k^2)O(k) + O\big(\tfrac{1}{k^2}\big)O(k) O(k) + O(1) \\
&= O(1)
\end{align*}
as $k \to \infty$.  

Now, the median value of $\{\!\!\{\ell_1,\ldots,\ell_m\}\!\!\}$ can shift by at most $M - m$ elements (counting multiplicity) upon including the lengths $\{\!\!\{\ell_{m+1},\ldots,\ell_M\}\!\!\}$.  
Lemmas~\ref{l:upperbound} and~\ref{l:lowerbound} provide upper and lower bounds for the multiplicities $f_k(x)$, respectively, that are $O(k)$, so each length has multiplicity $O(k)$ as well.  Moreover, $\Le\multi{ks}$ limits to a triangular distribution, so since $|\Le\multi{ks}| = \Theta(k^2)$, the lengths adjacent to $\eta(ks)$ on either side (which are all fixed distance $\delta$ appart) must have multiplicity $\Theta(k)$ (in particular, not just $O(k)$).  As such, since~$M - m = \Theta(k)$ by~\eqref{eq:Mmtheta}, 
\begin{equation*}
|\eta(ks + r) - \eta(ks)| = \frac{\Theta(k)}{\Theta(k)} \delta = O(1),
\end{equation*}
at which point we conclude that \eqref{eq:limcomp} holds for fixed $r$ for both $\eta$ and $\mu$.  
\end{proof}

Using the continuous probability distribution function $L$, we are now ready to compute the limits of $\mu(n)/n$ and $\eta(n)/n$ as $n \to \infty$, completing the proofs of Theorems~\ref{t:mean} and~\ref{t:median}, respectively.  

\begin{proof}[Proof of Theorem~\ref{t:mean}]
Uniform convergence of the normalized $S_k$ and~\eqref{eq:trimean} yield
\begin{equation*}
\lim_{k \to \infty} \int_0^1 xS_k(x) \,dx = \int_0^1 xL(x) \,dx = \frac{1+F}{3},
\end{equation*}
and upon inverting each transformation~$T_k$, we obtain 
\begin{equation*}
\tfrac{1}{3} T_k^{-1}(1 + F) = \tfrac{1}{3} \big[(1 + F)(b_k - a_k) + a_k\big] = \tfrac{1}{3} (a_k + b_k + c_k)
\end{equation*}
as the mean of the original distributions. 
Lemma~\ref{l:subseqconv} now implies
\begin{equation*}
\lim_{n \to \infty} \frac{\mu(n)}{n} = \lim_{k \to \infty} \frac{\mu(ks)}{ks} = \frac{1}{3} \left( \frac{1}{n_1} + \frac{1}{n_2} + \frac{1}{n_3} \right ),
\end{equation*}
which completes the proof.  
\end{proof}


\begin{proof}[Proof of Theorem~\ref{t:median}]
Applying~\eqref{eq:trimed}, first suppose $F \ge \frac{1}{2}$.  
Computing the preimage under $T_k$ and applying Lemma~\ref{l:subseqconv} yields
\begin{equation*}
\begin{split} 
\lim_{k \to \infty} \frac{\eta(ks)}{ks} & = \lim_{k \to \infty} T_k^{-1} \left(\sqrt{\frac{F}{2}}\right) \frac{1}{ks} = \lim_{k \to \infty} \frac{b_k - a_k}{ks} \sqrt{\frac{F}{2}} + \frac{a_k}{ks} \\ 
& = \frac{1}{n_1} \sqrt{\frac{F}{2}} + \frac{1}{n_3}\left(1 - \sqrt{\frac{F}{2}}\right).
\end{split}
\end{equation*}
The case $F \le \frac{1}{2}$ is similar.  

The case $F = \frac{1}{2}$ warrants special attention. For such numerical semigroups, the above demonstrated equalities reveal
\begin{equation*}
\gamma(ks) = \{\mu(ks)\} = \{\eta(ks)\} = \big\{ks/n_2\big\},
\end{equation*}
from which we obtain~\eqref{eq:everything2}.
\end{proof}

\section{Realizable median constants}
\label{sec:realizemedian}

The potentially irrational quantities that appear in the formula~\eqref{eq:medianasymptotic} for the asymptotic median factorization length raise interesting questions.  Is it possible for~\eqref{eq:medianasymptotic} to be rational?  What sort of quadratic irrational numbers can be obtained?  We provide partial answers here, exhibiting (i) infinitely many numerical semigroups whose median constant is rational (Theorem~\ref{t:rationalmedian}) and (ii) infinitely many numerical semigroups whose median constant is an irrational element of $\QQ(\sqrt{d})$, in which $d$ is any square-free positive integer (Theorem~\ref{t:irrationalmedian}).  Both constructions consist of numerical semigroups minimally generated by arithmetic sequences, a family of semigroups known in the literature for being particularly well-behaved \cite{setsoflength,diophantine,catenaryarith}.  

\begin{Theorem}\label{t:rationalmedian}
There are infinitely many minimally generated numerical semigroups 
semigroups $S = \langle n_1,n_2,n_3 \rangle$ so that $\eta(n)/n$ tends to a rational number.
\end{Theorem}

\begin{proof}
Let $a,b,c$ denote a primitive Pythagorean triple; that is, $\gcd(a,b,c) = 1$ and $a^2 + b^2 = c^2$.  It is well-known that there are infinitely many such triples
and that $a,b \geq 3$.  Suppose that $a > b$.  Let
\begin{equation*}
n_1 = a^2 - b^2 , \qquad n_2 = a^2, \qquad n^3 = a^2 + b^2 = c^2,
\end{equation*}
and $S = \<n_1, n_2, n_3\>$.
Then the fulcrum constant is
\begin{equation*}
F = \dfrac{n_1(n_3 - n_2)}{n_2(n_3 - n_1)}
= \frac{(a^2 - b^2)b^2}{a^2(2b^2)}
= \frac{a^2 - b^2}{2a^2} < \frac{1}{2}.
\end{equation*}
Thus, we are in the first case of Theorem~\ref{t:median}. To show that the limit of $\eta(n)/n$ is rational, 
we verify that the expression 
\begin{equation*}
\frac{(n_2 - n_1)(n_3 - n_1)}{2 n_2 n_3}
= \frac{b^2(2b^2)}{2a^2 c^2}
= \frac{b^4}{a^2 c^2}
= \left( \frac{b^2}{ac} \right)^2
\end{equation*}
under the radical in~\eqref{eq:medianasymptotic} is the square of a rational number.  

To complete the proof, notice that $n_1, n_2, n_3$ is a minimal generating set since it forms an arithmetic sequence whose minimal element $a^2 - b^2$ is relatively prime to the step size $b^2$ \cite{setsoflength}.  
\end{proof}


It is a folklore theorem that for any $\alpha \in \RR \setminus \ZZ$, the sequence
$\lfloor n \alpha \rfloor$ contains infinitely many prime numbers.  According to \cite[p.~240]{Harman},
this was first observed either by Heilbronn, as asserted by Vinogradov \cite[p.~180]{Vinogradov}, or by Fogels,
as suggested by Leitmann and Wolke \cite{Leitmann}.
In particular, the number of primes at most $x$
of the form $\lfloor n \alpha \rfloor$ is asymptotic to $\pi(x)/\alpha$, in which $\pi(x)$ denotes the number of primes at most $x$.  
For algebraic $\alpha$, one
can be more explicit:  the counting function is $\pi(x)/\alpha + O(x^{0.875+\epsilon})$ for every $\epsilon > 0$, in which
the implied constant depends on $\alpha$ and $\epsilon$ \cite{Begunts}.

\begin{Theorem}\label{t:irrationalmedian}
If $d \geq 2$ is a square-free positive integer, then there 
are infinitely many minimally generated numerical
semigroups $S = \langle n_1,n_2,n_3 \rangle$ so that
$\eta(n)/n$ tends to an irrational element of $\QQ(\sqrt{d})$.
\end{Theorem}

\begin{proof}
Let $d$ be a square-free positive integer and let $t \in \NN$ be so that
\begin{equation}\label{eq:pt}
p = \lfloor  t\sqrt{d} \rfloor > \max\{2,d\} 
\end{equation}
is prime.  In particular, $p > d$ and hence $p \nmid d$.
Since $d$ is not a square,
\begin{equation*}
t^2 d = p^2 + \ell
\end{equation*}
for some $\ell \geq 1$ and
\begin{equation}\label{eq:p2p2ell}
p^2 < p^2+\ell < (p+1)^2.
\end{equation}
In particular,
\begin{equation}\label{eq:ell2p}
\ell < 2p + 1.
\end{equation}

We claim that $p\nmid \ell$.
Suppose toward a contradiction that $p|\ell$.
Then $p|t^2 d$.  Since $p\nmid d$, we have $p|t^2$ and hence $p|t$ because $p$ is prime.
Therefore,
\begin{equation*}
p\sqrt{d}-1 \leq t \sqrt{d} - 1 < \lfloor  t\sqrt{d} \rfloor  = p
\end{equation*}
and hence
\begin{equation*}
p(\sqrt{d}-1) \leq 1,
\end{equation*}
which is a contradiction because both factors on the left-hand side are greater than $1$.  Therefore, $p \nmid \ell$.

Define
\begin{equation*}
n_1 = p^2 - \ell  , \qquad n_2 =p^2, \qquad n_3 = \underbrace{p^2 + \ell }_{t^2d}
\end{equation*}
and $S = \langle n_1,n_2,n_3 \rangle$.
Then the fulcrum constant is
\begin{equation*}
F = \dfrac{n_1(n_3-n_2)}{n_2(n_3-n_1)}
= \frac{(p^2-\ell )\ell }{p^2(2\ell )}
= \frac{1}{2} \cdot \frac{p^2-\ell }{p^2} < \frac{1}{2}.
\end{equation*}
Thus, we are in the first case of Theorem~\ref{t:median}.  The radical in~\eqref{eq:medianasymptotic} is
\begin{equation*}
\sqrt{\frac{(n_2-n_1)(n_3 - n_1)}{2 n_2 n_3}}
= \sqrt{ \frac{\ell (2\ell )}{2 p^2 (p^2 + \ell )} }
= \frac{\ell }{p} \frac{1}{\sqrt{p^2+\ell }} 
= \frac{\ell}{pt\sqrt{d}}
= \frac{\ell }{ptd} \sqrt{d},
\end{equation*}
which is a nonrational element of $\QQ(\sqrt{d})$ since $\ell \neq 0$.

As in Theorem~\ref{t:rationalmedian}, the semigroup in question is generated by an arithmetic sequence wth relatively prime minimal element $p^2 - \ell$ and step size $\ell$, and as such is minimally generated \cite{setsoflength}.  
\end{proof}

\section{Numerical semigroups with 4 or more generators}
\label{sec:embdim4}

We close with an example that demonstrates the difficulty in generalizing the results presented here to more general numerical semigroups.  

\begin{Example}\label{e:embdim4}
Let $S = \<4, 5, 6, 7\>$ and $n = 1680$, whose factorization length multiplicity function is plotted in Figure~\ref{f:embdim4}.  Taking common differences reveals points of inflection at $n/6$ and $n/5$, a fact that remains true empirically as $n \to \infty$.  For larger $n$, the length multiplicity function appears to fit a piecwise quadratic function with boundaries at $n/6$ and $n/5$.  On the other hand, the peak of the length multiplicity function for $n = 2520$ in $S = \<5, 7, 8, 9, 11\>$ does not occur at $n/8$, and the two visible points of inflection do not appear to occur at $n/9$ and $n/7$.  
\end{Example}

\begin{figure}
\includegraphics[width=0.45\textwidth]{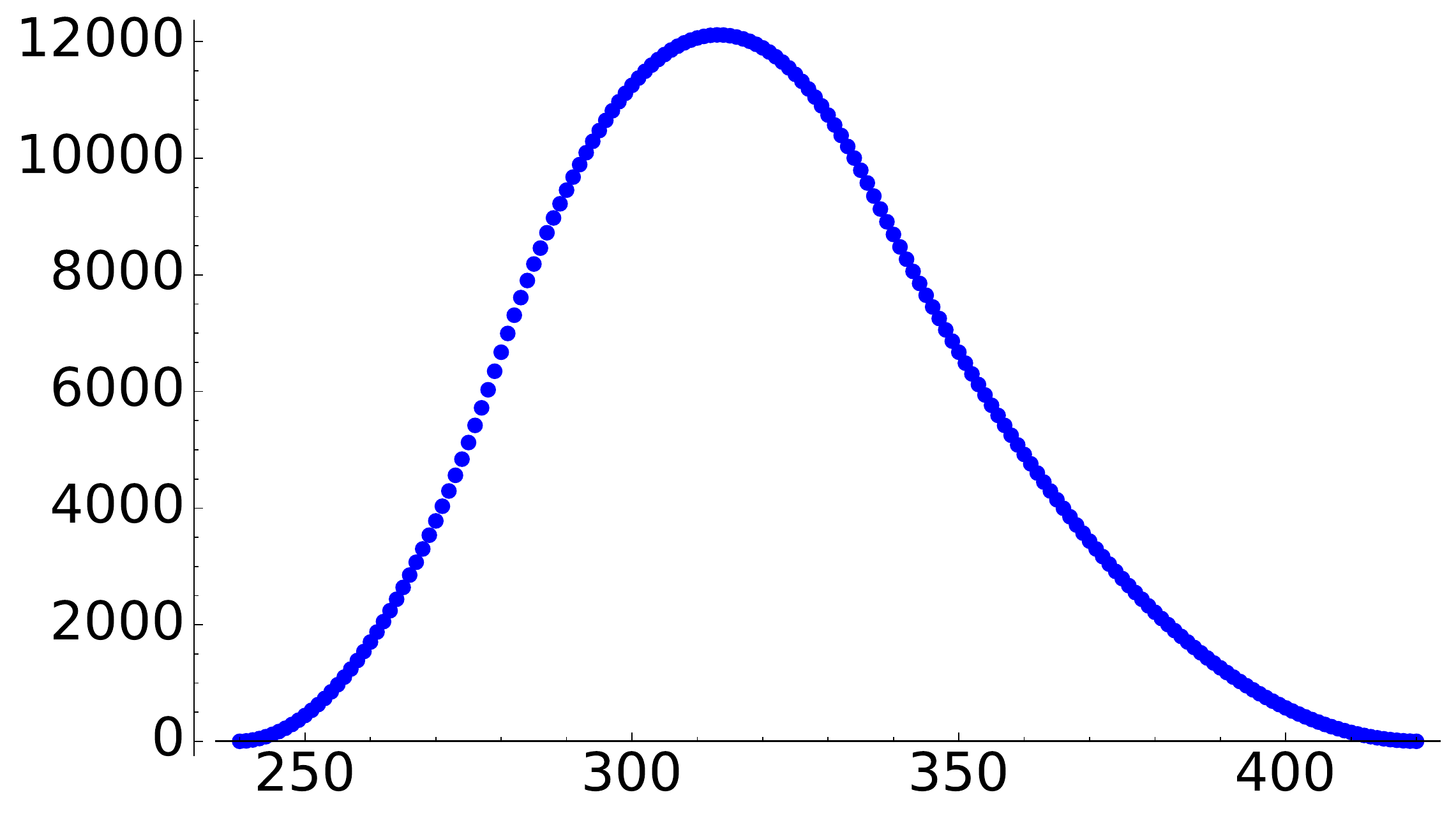}
\hspace{0.1in}
\includegraphics[width=0.45\textwidth]{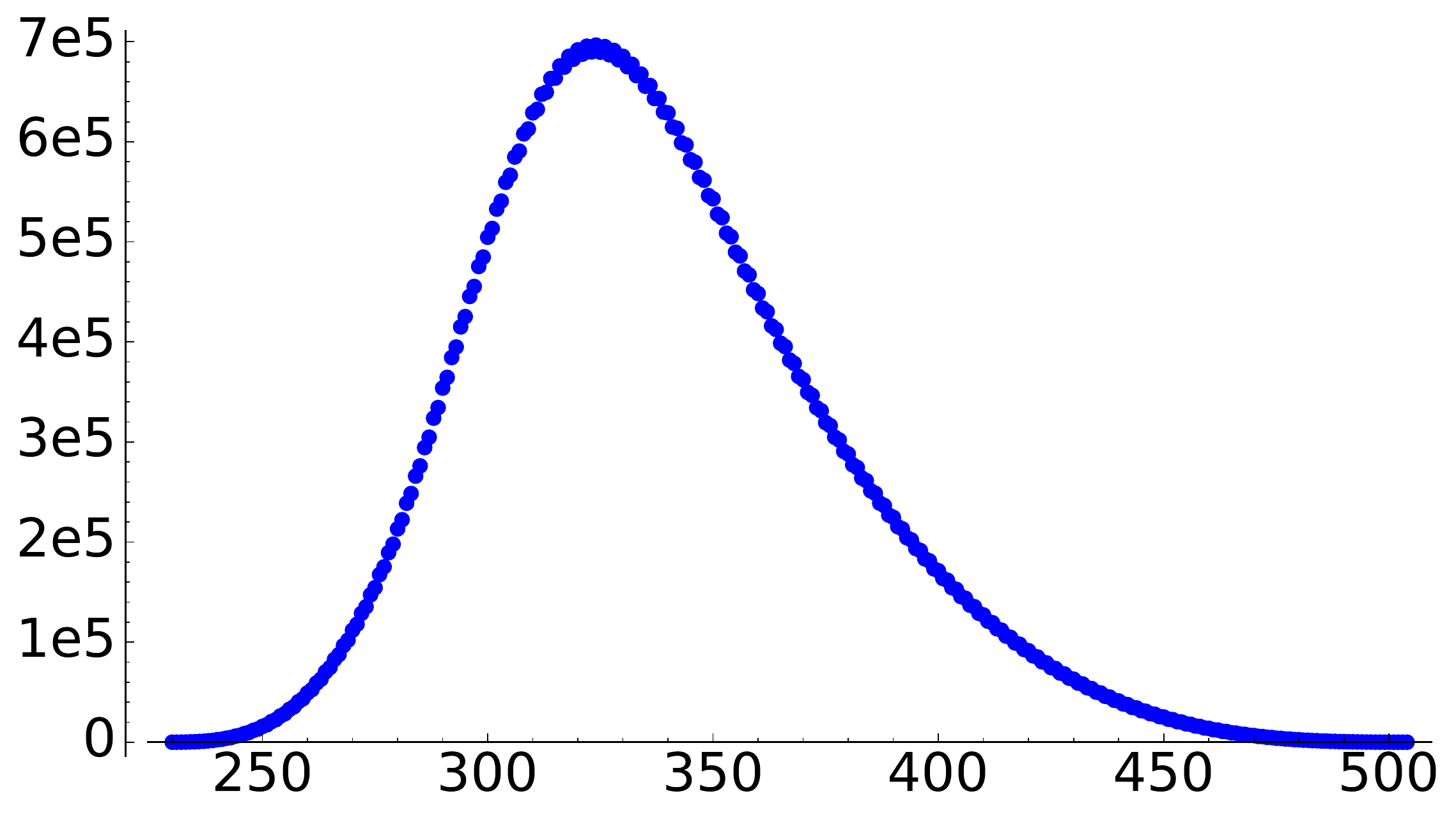}
\caption{Factorization length multiplicity functions of $1680 \in \<4, 5, 6, 7\>$ (left) and $2520 \in \<5, 7, 8, 9, 11\>$ (right) each plotted against factorization length.}
\label{f:embdim4}
\end{figure}

\section*{Acknowledgements}


Several figures in this manuscript, as well as data computed in the early stages of the project, made use of \texttt{Sage} \cite{sagemath} and the \texttt{GAP} package \texttt{numericalsgps} \cite{numericalsgpsgap}.  

\bibliography{FLD4AS1NS3G}
\bibliographystyle{amsplain}

\end{document}